\documentclass{amsart}
\usepackage{amsmath}
\usepackage{amsfonts}

\newtheorem{theorem}{Theorem}
\theoremstyle{plain}

\newtheorem{corollary}{Corollary}

\newtheorem{definition}{Definition}

\newtheorem{remark}{Remark}

\numberwithin{equation}{section}
\input{tcilatex}

\begin{document}
\title[Hadamard-type and Bullen-type inequalities]{Hadamard-type and
Bullen-type inequalities for Lipschitzian functions via fractional integrals}
\author{\.{I}mdat \.{I}\c{s}can}
\address{Department of Mathematics, Faculty of Sciences and Arts, Giresun
University, Giresun, Turkey}
\email{imdat.iscan@giresun.edu.tr}

\begin{abstract}
In this paper, the author establishes some Hadamard-type and Bullen-type
inequalities for Lipschitzian functions via Riemann Liouville fractional
integral. These results have some relationships with [K.-L. Tseng, S.-R.
Hwang and K.-C. Hsu, Hadamard-type and Bullen-type inequalities for
Lipschitzian functions and their applications, Computers and Mathematics
with Applications (2012). doi:10.1016/j.camwa.2011.12.076].
\end{abstract}

\maketitle

\section{Introduction}

Following inequality is well known in the literature as Hermite-Hadamard's
inequality:

\begin{theorem}
Let $f:I\subseteq \mathbb{R\rightarrow R}$ be a convex function defined on
the interval $I$ of real numbers and $a,b\in I$ with $a<b$. The following
double inequality holds%
\begin{equation}
f\left( \frac{a+b}{2}\right) \leq \frac{1}{b-a}\dint\limits_{a}^{b}f(x)dx%
\leq \frac{f(a)+f(b)}{2}\text{.}  \label{1-1}
\end{equation}
\end{theorem}

In \cite{THD10}, Tseng et al. established the following Hadamard-type
inequality which refines the inequality (\ref{1-1}).

\begin{theorem}
Suppose that $f:\left[ a,b\right] \mathbb{\rightarrow R}$ is a convex
function on $\left[ a,b\right] $. Then we have the inequalities:%
\begin{eqnarray}
f\left( \frac{a+b}{2}\right) &\leq &\frac{1}{2}\left[ f\left( \frac{3a+b}{4}%
\right) +f\left( \frac{a+3b}{4}\right) \right]  \notag \\
&\leq &\frac{1}{b-a}\dint\limits_{a}^{b}f(x)dx  \label{1-1a} \\
&\leq &\frac{1}{2}\left[ f\left( \frac{a+b}{2}\right) +\frac{f(a)+f(b)}{2}%
\right] \leq \frac{f(a)+f(b)}{2}.  \notag
\end{eqnarray}
\end{theorem}

The third inequality in (\ref{1-1a}) is known in the literature as Bullen's
inequality.

In what follows we recall the following definition.

\begin{definition}
A function $f:I\subseteq \mathbb{R\rightarrow R}$ is called an $M$%
-Lipschitzian function on the interval $I$ of real numbers with $M\geq 0$,
if 
\begin{equation*}
\left\vert f(x)-f(y)\right\vert \leq M\left\vert x-y\right\vert
\end{equation*}%
for all $x,y\in I.$
\end{definition}

For some recent results connected with Hermite-Hadamard type integral
inequalities for Lipschitzian functions, see \cite{DCK00,HHT13,THH12,YT01}.

In \cite{THH12}, Tseng et al. established some Hadamard-type and Bullen-type
inequalities for Lipschitzian functions as follows

\begin{theorem}
\label{1.1}Let $I$ be an interval in $%
\mathbb{R}
$, $a\leq A\leq B\leq b$ in $I$, $V=(1-\alpha )a+\alpha b$, $\alpha \in %
\left[ 0,1\right] $ and let $f:I\mathbb{\rightarrow R}$ be an $L$%
-Lipschitzian function with $L\geq 0.$ Then we have the inequality%
\begin{equation}
\left\vert \alpha f(A)+(1-\alpha )f(B)-\frac{1}{b-a}\dint%
\limits_{a}^{b}f(x)dx\right\vert \leq \frac{LV_{\alpha }(A,B)}{b-a},
\label{1-2}
\end{equation}%
where%
\begin{eqnarray*}
&&V_{\alpha }(A,B) \\
&=&\left\{ 
\begin{array}{c}
\left( A-a\right) ^{2}-\left( A-V\right) ^{2}+\left( B-V\right) ^{2}+\left(
b-B\right) ^{2}, \\ 
a\leq V\leq A\leq B\leq b, \\ 
\left( A-a\right) ^{2}+\left( V-A\right) ^{2}+\left( B-V\right) ^{2}+\left(
b-B\right) ^{2}, \\ 
a\leq A\leq V\leq B\leq b, \\ 
\left( A-a\right) ^{2}+\left( V-A\right) ^{2}+\left( b-B\right) ^{2}-\left(
V-B\right) ^{2}, \\ 
a\leq A\leq B\leq V\leq b%
\end{array}%
\right. .
\end{eqnarray*}
\end{theorem}

\begin{theorem}
\label{1.2}Let $I$ be an interval in $%
\mathbb{R}
$, $a\leq A\leq B\leq C\leq b$ in $I$, $V_{1}=(1-\alpha )a+\alpha b$, $%
V_{2}=\gamma a+\left( \alpha +\beta \right) b$, $\alpha ,\beta ,\gamma \in %
\left[ 0,1\right] $, $\alpha +\beta +\gamma =1$, and let $f:I\mathbb{%
\rightarrow R}$ be an $L$-Lipschitzian function with $L\geq 0.$ Then we have
the inequality%
\begin{equation}
\left\vert \alpha f(A)+\beta f(B)+\gamma f\left( C\right) -\frac{1}{b-a}%
\dint\limits_{a}^{b}f(x)dx\right\vert \leq \frac{LV_{\alpha ,\beta ,\gamma
}(A,B,C)}{b-a},  \label{1-3}
\end{equation}%
where $V_{\alpha ,\beta ,\gamma }$ is defined as in \cite[Section 3]{THH12}.
\end{theorem}

We give some necessary definitions and mathematical preliminaries of
fractional calculus theory which are used throughout this paper.

\begin{definition}
Let $f\in L\left[ a,b\right] $. The Riemann-Liouville integrals $%
J_{a^{+}}^{\alpha }f$ and $J_{b^{-}}^{\alpha }f$ of oder $\alpha >0$ with $%
a\geq 0$ are defined by

\begin{equation*}
J_{a^{+}}^{\alpha }f(x)=\frac{1}{\Gamma (\alpha )}\dint\limits_{a}^{x}\left(
x-t\right) ^{\alpha -1}f(t)dt,\ x>a
\end{equation*}

and

\begin{equation*}
J_{b^{-}}^{\alpha }f(x)=\frac{1}{\Gamma (\alpha )}\dint\limits_{x}^{b}\left(
t-x\right) ^{\alpha -1}f(t)dt,\ x<b
\end{equation*}%
respectively, where $\Gamma (\alpha )$ is the Gamma function defined by $%
\Gamma (\alpha )=$ $\dint\limits_{0}^{\infty }e^{-t}t^{\alpha -1}dt$ and $%
J_{a^{+}}^{0}f(x)=J_{b^{-}}^{0}f(x)=f(x).$
\end{definition}

In the case of $\alpha =1$, the fractional integral reduces to the classical
integral. For some recent results connected with fractional integral
inequalities, see \cite{D10,SO12,SSYB11,ZFW13}.

The aim of this paper is to establish some Hadamard-type and Bullen-type
inequalities for Lipschitzian functions via Riemann--Liouville fractional
integral.

\section{Hadamard-type inequalities for Lipschitzian functions via
fractional integrals}

Throughout this section, let $I$ be an interval in $%
\mathbb{R}
$, $a\leq x\leq y\leq b$ in $I$ and let $f:I\rightarrow 
\mathbb{R}
$ be an $M$-Lipschitzian function. In the next theorem, let $\lambda \in %
\left[ 0,1\right] $, $V=(1-\lambda )a+\lambda b$, and $V_{\alpha ,\lambda }$%
, $\alpha >0$, as follows:

(1) If $a\leq V\leq x\leq y\leq b$, then%
\begin{equation*}
V_{\alpha ,\lambda }(x,y)=\left( V-a\right) ^{\alpha }\left[ \frac{x-a}{%
\alpha }-\frac{V-a}{\alpha +1}\right] +\frac{2\left( b-y\right) ^{\alpha +1}%
}{\alpha \left( \alpha +1\right) }+\left( b-V\right) ^{\alpha }\left[ \frac{%
b-V}{\alpha +1}-\frac{b-y}{\alpha }\right]
\end{equation*}

(2) If $a\leq x\leq V\leq y\leq b$, then%
\begin{eqnarray*}
V_{\alpha ,\lambda }(x,y) &=&\frac{2\left( x-a\right) ^{\alpha +1}}{\alpha
\left( \alpha +1\right) }+\left( V-a\right) ^{\alpha }\left[ \frac{V-a}{%
\alpha +1}-\frac{x-a}{\alpha }\right] \\
&&+\frac{2\left( b-y\right) ^{\alpha +1}}{\alpha \left( \alpha +1\right) }%
+\left( b-V\right) ^{\alpha }\left[ \frac{b-V}{\alpha +1}-\frac{b-y}{\alpha }%
\right] .
\end{eqnarray*}

(3) If $a\leq x\leq y\leq V\leq b$, then%
\begin{equation*}
V_{\alpha ,\lambda }(x,y)=\frac{2\left( x-a\right) ^{\alpha +1}}{\alpha
\left( \alpha +1\right) }+\left( V-a\right) ^{\alpha }\left[ \frac{V-a}{%
\alpha +1}-\frac{x-a}{\alpha }\right] +\left( b-V\right) ^{\alpha }\left[ 
\frac{b-y}{\alpha }-\frac{b-V}{\alpha +1}\right] .
\end{equation*}

\begin{theorem}
\label{2.1}Let $x,y,\alpha ,\lambda ,V,V_{\alpha ,\lambda }$ and the
function $f$ be defined as above. Then we have the inequality for fractional
integrals%
\begin{equation}
\left\vert \lambda ^{\alpha }f(x)+(1-\lambda )^{\alpha }f(y)-\frac{\Gamma
(\alpha +1)}{\left( b-a\right) ^{\alpha }}\left[ J_{V-}^{\alpha
}f(a)+J_{V+}^{\alpha }f(b)\right] \right\vert \leq \frac{\alpha MV_{\alpha
,\lambda }(x,y)}{\left( b-a\right) ^{\alpha }}.  \label{2-1}
\end{equation}
\end{theorem}

\begin{proof}
Using the hypothesis of $f$, we have the following inequality%
\begin{eqnarray}
&&\left\vert \lambda ^{\alpha }f(x)+(1-\lambda )^{\alpha }f(y)-\frac{\Gamma
(\alpha +1)}{\left( b-a\right) ^{\alpha }}\left[ J_{V-}^{\alpha
}f(a)+J_{V+}^{\alpha }f(b)\right] \right\vert  \notag \\
&=&\frac{\alpha }{\left( b-a\right) ^{\alpha }}\left\vert
\dint\limits_{a}^{V}\left[ f(x)-f(t)\right] \left( t-a\right) ^{\alpha
-1}dt+\dint\limits_{V}^{b}\left[ f(y)-f(t)\right] \left( b-t\right) ^{\alpha
-1}dt\right\vert  \notag \\
&\leq &\frac{\alpha }{\left( b-a\right) ^{\alpha }}\left[ \dint%
\limits_{a}^{V}\left\vert f(x)-f(t)\right\vert \left( t-a\right) ^{\alpha
-1}dt+\dint\limits_{V}^{b}\left\vert f(y)-f(t)\right\vert \left( b-t\right)
^{\alpha -1}dt\right]  \notag \\
&\leq &\frac{\alpha M}{\left( b-a\right) ^{\alpha }}\left[
\dint\limits_{a}^{V}\left\vert x-t\right\vert \left( t-a\right) ^{\alpha
-1}dt+\dint\limits_{V}^{b}\left\vert y-t\right\vert \left( b-t\right)
^{\alpha -1}dt\right] .  \label{2-1a}
\end{eqnarray}%
Now using simple calculations, we obtain the following identities $%
\int_{a}^{V}\left\vert x-t\right\vert \left( t-a\right) ^{\alpha -1}dt$ and $%
\int_{V}^{b}\left\vert y-t\right\vert \left( b-t\right) ^{\alpha -1}dt.$

(1) If $a\leq V\leq x\leq y\leq b$, then%
\begin{equation*}
\dint\limits_{a}^{V}\left\vert x-t\right\vert \left( t-a\right) ^{\alpha
-1}dt=\left( V-a\right) ^{\alpha }\left[ \frac{x-a}{\alpha }-\frac{V-a}{%
\alpha +1}\right]
\end{equation*}%
and%
\begin{equation*}
\dint\limits_{V}^{b}\left\vert y-t\right\vert \left( b-t\right) ^{\alpha
-1}dt=\frac{2\left( b-y\right) ^{\alpha +1}}{\alpha \left( \alpha +1\right) }%
+\left( b-V\right) ^{\alpha }\left[ \frac{b-V}{\alpha +1}-\frac{b-y}{\alpha }%
\right] .
\end{equation*}%
(2) If $a\leq x\leq V\leq y\leq b$, then%
\begin{equation*}
\dint\limits_{a}^{V}\left\vert x-t\right\vert \left( t-a\right) ^{\alpha
-1}dt=\frac{2\left( x-a\right) ^{\alpha +1}}{\alpha \left( \alpha +1\right) }%
+\left( V-a\right) ^{\alpha }\left[ \frac{V-a}{\alpha +1}-\frac{x-a}{\alpha }%
\right]
\end{equation*}%
and%
\begin{equation*}
\dint\limits_{V}^{b}\left\vert y-t\right\vert \left( b-t\right) ^{\alpha
-1}dt=\frac{2\left( b-y\right) ^{\alpha +1}}{\alpha \left( \alpha +1\right) }%
+\left( b-V\right) ^{\alpha }\left[ \frac{b-V}{\alpha +1}-\frac{b-y}{\alpha }%
\right] .
\end{equation*}%
(3) If $a\leq x\leq y\leq V\leq b$, then%
\begin{equation*}
\dint\limits_{a}^{V}\left\vert x-t\right\vert \left( t-a\right) ^{\alpha
-1}dt=\frac{2\left( x-a\right) ^{\alpha +1}}{\alpha \left( \alpha +1\right) }%
+\left( V-a\right) ^{\alpha }\left[ \frac{V-a}{\alpha +1}-\frac{x-a}{\alpha }%
\right]
\end{equation*}%
and%
\begin{equation*}
\dint\limits_{V}^{b}\left\vert y-t\right\vert \left( b-t\right) ^{\alpha
-1}dt=\left( b-V\right) ^{\alpha }\left[ \frac{b-y}{\alpha }-\frac{b-V}{%
\alpha +1}\right] .
\end{equation*}%
Using the inequality (\ref{2-1a}) and the above identities $%
\int_{a}^{V}\left\vert x-t\right\vert \left( t-a\right) ^{\alpha -1}dt$ and $%
\int_{V}^{b}\left\vert y-t\right\vert \left( b-t\right) ^{\alpha -1}dt$, we
derive the inequality (\ref{2-1}). This completes the proof.
\end{proof}

Under the assumptions of Theorem \ref{2.1}, we have the following
corollaries and remarks:

\begin{remark}
In Theorem \ref{2.1}, if we take $\alpha =1$, then the inequality (\ref{2-1}%
) reduces the inequality (\ref{1-2}) in Theorem \ref{1.1}.
\end{remark}

\begin{corollary}
\begin{enumerate}
\item In Theorem \ref{2.1}, let $\delta \in \left[ \frac{1}{2},1\right] $, $%
x=\delta a+(1-\delta )b$ and $y=(1-\delta )a+\delta b$. Then, we have the
inequality%
\begin{equation*}
\left\vert \lambda ^{\alpha }f(\delta a+(1-\delta )b)+(1-\lambda )^{\alpha
}f((1-\delta )a+\delta b)-\frac{\Gamma (\alpha +1)}{\left( b-a\right)
^{\alpha }}\left[ J_{V-}^{\alpha }f(a)+J_{V+}^{\alpha }f(b)\right]
\right\vert
\end{equation*}%
\begin{equation}
\leq \frac{ML\left( \alpha ,\lambda ,\delta \right) \left( b-a\right) }{%
\alpha +1}  \label{2-1b}
\end{equation}%
where%
\begin{eqnarray*}
&&L\left( \alpha ,\lambda ,\delta \right) \\
&=&\left\{ 
\begin{array}{c}
\begin{array}{c}
\lambda ^{\alpha }\left[ \left( 1-\delta \right) \left( 1+\alpha \right)
-\lambda \alpha \right] +2\left( 1-\delta \right) ^{\alpha +1}+\left(
1-\lambda \right) ^{\alpha }\left[ \left( 1-\lambda \right) \alpha -\left(
1-\delta \right) \left( 1+\alpha \right) \right] , \\ 
\lambda \leq 1-\delta%
\end{array}
\\ 
4\left( 1-\delta \right) ^{\alpha +1}+\lambda ^{\alpha }\left[ \lambda
\alpha -\left( 1-\delta \right) \left( 1+\alpha \right) \right] +\left(
1-\lambda \right) ^{\alpha }\left[ \left( 1-\lambda \right) \alpha -\left(
1-\delta \right) \left( 1+\alpha \right) \right] , \\ 
1-\delta \leq \lambda \leq \delta \\ 
2\left( 1-\delta \right) ^{\alpha +1}+\lambda ^{\alpha }\left[ \lambda
\alpha -\left( 1-\delta \right) \left( 1+\alpha \right) \right] +\left(
1-\lambda \right) ^{\alpha }\left[ \left( 1-\delta \right) \left( 1+\alpha
\right) -\left( 1-\lambda \right) \alpha \right] , \\ 
\delta \leq \lambda .%
\end{array}%
\right.
\end{eqnarray*}

\item In Theorem \ref{2.1}, if we take $x=y=V$, then we have the inequality%
\begin{eqnarray}
&&\left\vert \left[ \lambda ^{\alpha }+(1-\lambda )^{\alpha }\right] f(x)-%
\frac{\Gamma (\alpha +1)}{\left( b-a\right) ^{\alpha }}\left[ J_{V-}^{\alpha
}f(a)+J_{V+}^{\alpha }f(b)\right] \right\vert  \label{2-1c} \\
&\leq &M\frac{\left( x-a\right) ^{\alpha +1}+\left( b-x\right) ^{\alpha +1}}{%
\left( \alpha +1\right) \left( b-a\right) ^{\alpha }}.  \notag
\end{eqnarray}
\end{enumerate}
\end{corollary}

\begin{corollary}
We have the following weighted Hadamard-type inequalities for Lipschitzian
functions via Rieamnn-liouville fractional integrals

\begin{enumerate}
\item[(1)] In the inequality (\ref{2-1b}), if we take $\delta =1$, then we
have%
\begin{eqnarray*}
&&\left\vert \lambda ^{\alpha }f(a)+(1-\lambda )^{\alpha }f((b)-\frac{\Gamma
(\alpha +1)}{\left( b-a\right) ^{\alpha }}\left[ J_{V-}^{\alpha
}f(a)+J_{V+}^{\alpha }f(b)\right] \right\vert \\
&\leq &\alpha M\left( b-a\right) \frac{\lambda ^{\alpha +1}+\left( 1-\lambda
\right) ^{\alpha +1}}{\alpha +1},
\end{eqnarray*}%
in this inequality, specially if we choose $\lambda =\frac{x-a}{b-a}$ for $%
x\in \left[ a,b\right] $, then 
\begin{eqnarray*}
&&\left\vert \frac{\left( x-a\right) ^{\alpha }f(a)+\left( b-x\right)
^{\alpha }f((b)}{\left( b-a\right) ^{\alpha }}-\frac{\Gamma (\alpha +1)}{%
\left( b-a\right) ^{\alpha }}\left[ J_{x-}^{\alpha }f(a)+J_{x+}^{\alpha }f(b)%
\right] \right\vert \\
&\leq &\alpha M\frac{\left( x-a\right) ^{\alpha +1}+\left( b-x\right)
^{\alpha +1}}{\left( \alpha +1\right) \left( b-a\right) ^{\alpha }},
\end{eqnarray*}

\item[(2)] In the inequality (\ref{2-1c}), if we take $x=\delta a+(1-\delta
)b,\ \delta \in \left[ 0,1\right] $, then 
\begin{eqnarray*}
&&\left\vert \left[ \lambda ^{\alpha }+(1-\lambda )^{\alpha }\right]
f(\delta a+(1-\delta )b)-\frac{\Gamma (\alpha +1)}{\left( b-a\right)
^{\alpha }}\left[ J_{V-}^{\alpha }f(a)+J_{V+}^{\alpha }f(b)\right]
\right\vert \\
&\leq &M\left( b-a\right) \frac{\delta ^{\alpha +1}+\left( 1-\delta \right)
^{\alpha +1}}{\left( \alpha +1\right) },
\end{eqnarray*}%
in this inequality, specially if we choose $\lambda =\frac{1}{2}$, then 
\begin{eqnarray*}
&&\left\vert f(\delta a+(1-\delta )b)-\frac{2^{\alpha -1}\Gamma (\alpha +1)}{%
\left( b-a\right) ^{\alpha }}\left[ J_{\left( \frac{a+b}{2}\right)
-}^{\alpha }f(a)+J_{\left( \frac{a+b}{2}\right) +}^{\alpha }f(b)\right]
\right\vert \\
&\leq &2^{\alpha -1}M\left( b-a\right) \frac{\delta ^{\alpha +1}+\left(
1-\delta \right) ^{\alpha +1}}{\left( \alpha +1\right) },
\end{eqnarray*}

\item[(3)] In the inequality (\ref{2-1c}), if we take $\lambda =\frac{1}{2}$
and $\delta =\frac{3}{4}$ then 
\begin{eqnarray*}
&&\left\vert \frac{1}{2}\left[ f\left( \frac{3a+b}{4}\right) +f\left( \frac{%
a+3b}{4}\right) \right] -\frac{2^{\alpha -1}\Gamma (\alpha +1)}{\left(
b-a\right) ^{\alpha }}\left[ J_{\left( \frac{a+b}{2}\right) -}^{\alpha
}f(a)+J_{\left( \frac{a+b}{2}\right) +}^{\alpha }f(b)\right] \right\vert \\
&\leq &M\left( b-a\right) \frac{1+2^{\alpha -1}(\alpha -1)}{2^{\alpha
+1}\left( \alpha +1\right) }
\end{eqnarray*}
\end{enumerate}
\end{corollary}

\section{Bullen -type inequalities for Lipschitzian functions via fractional
integrals}

Throughout this section, let $I$ be an interval in $%
\mathbb{R}
$, $a\leq x\leq y\leq z\leq b$ in $I$ and $f:I\rightarrow 
\mathbb{R}
$ be an $M$-lipschitzian function. In the next theorem, let $\lambda +\eta
+\mu =1$, $\lambda ,\eta ,\mu \in \left[ 0,1\right] $, $V_{1}=(1-\lambda
)a+\lambda b$, $V_{2}=\mu a+\left( \lambda +\eta \right) b$, and define $%
V_{\alpha ,\lambda ,\eta ,\mu }$, $\alpha >0$, as follows:

\begin{enumerate}
\item If $V_{1}\leq V_{2}\leq x\leq y\leq z$ or $V_{1}\leq x\leq V_{2}\leq
y\leq z$, then%
\begin{eqnarray*}
V_{\alpha ,\lambda ,\eta ,\mu }(x,y,z) &=&\left( V_{1}-a\right) ^{\alpha } 
\left[ \frac{x-a}{\alpha }-\frac{V_{1}-a}{\alpha +1}\right] +\left(
V_{2}-V_{1}\right) ^{\alpha }\left[ \frac{y-V_{2}}{\alpha }+\frac{V_{2}-V_{1}%
}{\alpha +1}\right] \\
&&+\frac{2\left( b-z\right) ^{\alpha +1}}{\alpha \left( \alpha +1\right) }%
+\left( b-V_{2}\right) ^{\alpha }\left[ \frac{b-V_{2}}{\alpha +1}-\frac{b-z}{%
\alpha }\right] .
\end{eqnarray*}

\item If $V_{1}\leq x\leq y\leq V_{2}\leq z$, then 
\begin{eqnarray*}
V_{\alpha ,\lambda ,\eta ,\mu }(x,y,z) &=&\left( V_{1}-a\right) ^{\alpha } 
\left[ \frac{x-a}{\alpha }-\frac{V_{1}-a}{\alpha +1}\right] +\frac{2\left(
V_{2}-y\right) ^{\alpha +1}}{\alpha \left( \alpha +1\right) }+\left(
V_{2}-V_{1}\right) ^{\alpha }\left[ \frac{V_{2}-V_{1}}{\alpha +1}-\frac{%
V_{2}-y}{\alpha }\right] \\
&&+\frac{2\left( b-z\right) ^{\alpha +1}}{\alpha \left( \alpha +1\right) }%
+\left( b-V_{2}\right) ^{\alpha }\left[ \frac{b-V_{2}}{\alpha +1}-\frac{b-z}{%
\alpha }\right] .
\end{eqnarray*}

\item If $V_{1}\leq x\leq y\leq z\leq V_{2}$, then 
\begin{eqnarray*}
V_{\alpha ,\lambda ,\eta ,\mu }(x,y,z) &=&\left( V_{1}-a\right) ^{\alpha } 
\left[ \frac{x-a}{\alpha }-\frac{V_{1}-a}{\alpha +1}\right] +\frac{2\left(
V_{2}-y\right) ^{\alpha +1}}{\alpha \left( \alpha +1\right) } \\
&&+\left( V_{2}-V_{1}\right) ^{\alpha }\left[ \frac{V_{2}-V_{1}}{\alpha +1}-%
\frac{V_{2}-y}{\alpha }\right] +\left( b-V_{2}\right) ^{\alpha }\left[ \frac{%
b-z}{\alpha }-\frac{b-V_{2}}{\alpha +1}\right] .
\end{eqnarray*}

\item If $x\leq V_{1}\leq V_{2}\leq y\leq z$, then 
\begin{eqnarray*}
V_{\alpha ,\lambda ,\eta ,\mu }(x,y,z) &=&\frac{2\left( x-a\right) ^{\alpha
+1}}{\alpha \left( \alpha +1\right) }+\left( V_{1}-a\right) ^{\alpha }\left[ 
\frac{V_{1}-a}{\alpha +1}-\frac{x-a}{\alpha }\right] +\left(
V_{2}-V_{1}\right) ^{\alpha }\left[ \frac{y-V_{2}}{\alpha }+\frac{V_{2}-V_{1}%
}{\alpha +1}\right] \\
&&+\frac{2\left( b-z\right) ^{\alpha +1}}{\alpha \left( \alpha +1\right) }%
+\left( b-V_{2}\right) ^{\alpha }\left[ \frac{b-V_{2}}{\alpha +1}-\frac{b-z}{%
\alpha }\right] .
\end{eqnarray*}

\item If $x\leq V_{1}\leq y\leq V_{2}\leq z$, then 
\begin{eqnarray*}
V_{\alpha ,\lambda ,\eta ,\mu }(x,y,z) &=&\frac{2\left( x-a\right) ^{\alpha
+1}}{\alpha \left( \alpha +1\right) }+\left( V_{1}-a\right) ^{\alpha }\left[ 
\frac{V_{1}-a}{\alpha +1}-\frac{x-a}{\alpha }\right] +\frac{2\left(
V_{2}-y\right) ^{\alpha +1}}{\alpha \left( \alpha +1\right) } \\
&&+\left( V_{2}-V_{1}\right) ^{\alpha }\left[ \frac{V_{2}-V_{1}}{\alpha +1}-%
\frac{V_{2}-y}{\alpha }\right] +\frac{2\left( b-z\right) ^{\alpha +1}}{%
\alpha \left( \alpha +1\right) }+\left( b-V_{2}\right) ^{\alpha }\left[ 
\frac{b-V_{2}}{\alpha +1}-\frac{b-z}{\alpha }\right] .
\end{eqnarray*}

\item If $x\leq V_{1}\leq y\leq z\leq V_{2}$, then%
\begin{eqnarray*}
V_{\alpha ,\lambda ,\eta ,\mu }(x,y,z) &=&\frac{2\left( x-a\right) ^{\alpha
+1}}{\alpha \left( \alpha +1\right) }+\left( V_{1}-a\right) ^{\alpha }\left[ 
\frac{V_{1}-a}{\alpha +1}-\frac{x-a}{\alpha }\right] +\frac{2\left(
V_{2}-y\right) ^{\alpha +1}}{\alpha \left( \alpha +1\right) } \\
&&+\left( V_{2}-V_{1}\right) ^{\alpha }\left[ \frac{V_{2}-V_{1}}{\alpha +1}-%
\frac{V_{2}-y}{\alpha }\right] +\left( b-V_{2}\right) ^{\alpha }\left[ \frac{%
b-z}{\alpha }-\frac{b-V_{2}}{\alpha +1}\right] .
\end{eqnarray*}

\item If $x\leq y\leq V_{1}\leq V_{2}\leq z$, then%
\begin{eqnarray*}
V_{\alpha ,\lambda ,\eta ,\mu }(x,y,z) &=&\frac{2\left( x-a\right) ^{\alpha
+1}}{\alpha \left( \alpha +1\right) }+\left( V_{1}-a\right) ^{\alpha }\left[ 
\frac{V_{1}-a}{\alpha +1}-\frac{x-a}{\alpha }\right] +\left(
V_{2}-V_{1}\right) ^{\alpha }\left[ \frac{V_{2}-V_{1}}{\alpha +1}-\frac{%
V_{2}-y}{\alpha }\right] \\
&&+\frac{2\left( b-z\right) ^{\alpha +1}}{\alpha \left( \alpha +1\right) }%
+\left( b-V_{2}\right) ^{\alpha }\left[ \frac{b-V_{2}}{\alpha +1}-\frac{b-z}{%
\alpha }\right] .
\end{eqnarray*}

\item If $x\leq y\leq V_{1}\leq z\leq V_{2}$ or $x\leq y\leq z\leq V_{1}\leq
V_{2}$, then%
\begin{eqnarray*}
V_{\alpha ,\lambda ,\eta ,\mu }(x,y,z) &=&\frac{2\left( x-a\right) ^{\alpha
+1}}{\alpha \left( \alpha +1\right) }+\left( V_{1}-a\right) ^{\alpha }\left[ 
\frac{V_{1}-a}{\alpha +1}-\frac{x-a}{\alpha }\right] \\
&&+\left( V_{2}-V_{1}\right) ^{\alpha }\left[ \frac{V_{2}-V_{1}}{\alpha +1}-%
\frac{V_{2}-y}{\alpha }\right] +\left( b-V_{2}\right) ^{\alpha }\left[ \frac{%
b-z}{\alpha }-\frac{b-V_{2}}{\alpha +1}\right] .
\end{eqnarray*}%
\bigskip
\end{enumerate}

\begin{theorem}
\label{3.1}Let $x,y,z,\lambda ,\eta ,\mu ,V_{1},V_{2},V_{\alpha ,\lambda
,\eta ,\mu }$ and the function $f$ be defined as above. Then we have the
inequality%
\begin{equation*}
\left\vert \lambda ^{\alpha }f(x)+\eta ^{\alpha }f(y)+\mu ^{\alpha }f(z)-%
\frac{\Gamma (\alpha +1)}{\left( b-a\right) ^{\alpha }}\left[
J_{V_{1}-}^{\alpha }f(a)+J_{V_{1}+}^{\alpha }f(V_{2})+J_{V_{2}+}^{\alpha
}f(b)\right] \right\vert 
\end{equation*}%
\begin{equation}
\leq \frac{\alpha MV_{\alpha ,\lambda ,\eta ,\mu }(x,y,z)}{\left( b-a\right)
^{\alpha }}  \label{3-1}
\end{equation}
\end{theorem}

\begin{proof}
Using the hypothesis of $f$, we have the inequality%
\begin{equation*}
\left\vert \lambda ^{\alpha }f(x)+\eta ^{\alpha }f(y)+\mu ^{\alpha }f(z)-%
\frac{\Gamma (\alpha +1)}{\left( b-a\right) ^{\alpha }}\left[
J_{V_{1}-}^{\alpha }f(a)+J_{V_{1}+}^{\alpha }f(V_{2})+J_{V_{2}+}^{\alpha
}f(b)\right] \right\vert
\end{equation*}%
\begin{equation*}
=\frac{\alpha }{\left( b-a\right) ^{\alpha }}\left\vert
\dint\limits_{a}^{V_{1}}\left[ f(x)-f(t)\right] \left( t-a\right) ^{\alpha
-1}dt+\dint\limits_{V_{1}}^{V_{2}}\left[ f(y)-f(t)\right] \left(
V_{2}-t\right) ^{\alpha -1}dt+\dint\limits_{V_{2}}^{b}\left[ f(z)-f(t)\right]
\left( b-t\right) ^{\alpha -1}dt\right\vert
\end{equation*}%
\begin{equation*}
\leq \frac{\alpha }{\left( b-a\right) ^{\alpha }}\left[ \dint%
\limits_{a}^{V_{1}}\left\vert f(x)-f(t)\right\vert \left( t-a\right)
^{\alpha -1}dt+\dint\limits_{V_{1}}^{V_{2}}\left\vert f(y)-f(t)\right\vert
\left( V_{2}-t\right) ^{\alpha -1}dt+\dint\limits_{V_{2}}^{b}\left\vert
f(z)-f(t)\right\vert \left( b-t\right) ^{\alpha -1}dt\right]
\end{equation*}%
\begin{equation}
\leq \frac{\alpha M}{\left( b-a\right) ^{\alpha }}\left[ \dint%
\limits_{a}^{V_{1}}\left\vert x-t\right\vert \left( t-a\right) ^{\alpha
-1}dt+\dint\limits_{V_{1}}^{V_{2}}\left\vert y-t\right\vert \left(
V_{2}-t\right) ^{\alpha -1}dt+\dint\limits_{V_{2}}^{b}\left\vert
z-t\right\vert \left( b-t\right) ^{\alpha -1}dt\right] .  \label{3-1a}
\end{equation}%
Now, using simple calculations, we obtain the following identities $%
\int_{a}^{V_{1}}\left\vert x-t\right\vert \left( t-a\right) ^{\alpha -1}dt$, 
$\int_{V_{1}}^{V_{2}}\left\vert y-t\right\vert \left( V_{2}-t\right)
^{\alpha -1}dt$ and $\int_{V_{2}}^{b}\left\vert z-t\right\vert \left(
b-t\right) ^{\alpha -1}dt.$

\begin{enumerate}
\item If $V_{1}\leq V_{2}\leq x\leq y\leq z$ or $V_{1}\leq x\leq V_{2}\leq
y\leq z$, then we have%
\begin{eqnarray*}
\dint\limits_{a}^{V_{1}}\left\vert x-t\right\vert \left( t-a\right) ^{\alpha
-1}dt &=&\left( V_{1}-a\right) ^{\alpha }\left[ \frac{x-a}{\alpha }-\frac{%
V_{1}-a}{\alpha +1}\right] , \\
\dint\limits_{V_{1}}^{V_{2}}\left\vert y-t\right\vert \left( V_{2}-t\right)
^{\alpha -1}dt &=&\left( V_{2}-V_{1}\right) ^{\alpha }\left[ \frac{y-V_{2}}{%
\alpha }+\frac{V_{2}-V_{1}}{\alpha +1}\right]
\end{eqnarray*}%
and%
\begin{equation*}
\dint\limits_{V_{2}}^{b}\left\vert z-t\right\vert \left( b-t\right) ^{\alpha
-1}dt=\frac{2\left( b-z\right) ^{\alpha +1}}{\alpha \left( \alpha +1\right) }%
+\left( b-V_{2}\right) ^{\alpha }\left[ \frac{b-V_{2}}{\alpha +1}-\frac{b-z}{%
\alpha }\right] .
\end{equation*}

\item If $V_{1}\leq x\leq y\leq V_{2}\leq z$, then we have%
\begin{eqnarray*}
\dint\limits_{a}^{V_{1}}\left\vert x-t\right\vert \left( t-a\right) ^{\alpha
-1}dt &=&\left( V_{1}-a\right) ^{\alpha }\left[ \frac{x-a}{\alpha }-\frac{%
V_{1}-a}{\alpha +1}\right] , \\
\dint\limits_{V_{1}}^{V_{2}}\left\vert y-t\right\vert \left( V_{2}-t\right)
^{\alpha -1}dt &=&\frac{2\left( V_{2}-y\right) ^{\alpha +1}}{\alpha \left(
\alpha +1\right) }+\left( V_{2}-V_{1}\right) ^{\alpha }\left[ \frac{%
V_{2}-V_{1}}{\alpha +1}-\frac{V_{2}-y}{\alpha }\right]
\end{eqnarray*}%
and%
\begin{equation*}
\dint\limits_{V_{2}}^{b}\left\vert z-t\right\vert \left( b-t\right) ^{\alpha
-1}dt=\frac{2\left( b-z\right) ^{\alpha +1}}{\alpha \left( \alpha +1\right) }%
+\left( b-V_{2}\right) ^{\alpha }\left[ \frac{b-V_{2}}{\alpha +1}-\frac{b-z}{%
\alpha }\right] .
\end{equation*}

\item If $V_{1}\leq x\leq y\leq z\leq V_{2}$, then we have%
\begin{eqnarray*}
\dint\limits_{a}^{V_{1}}\left\vert x-t\right\vert \left( t-a\right) ^{\alpha
-1}dt &=&\left( V_{1}-a\right) ^{\alpha }\left[ \frac{x-a}{\alpha }-\frac{%
V_{1}-a}{\alpha +1}\right] , \\
\dint\limits_{V_{1}}^{V_{2}}\left\vert y-t\right\vert \left( V_{2}-t\right)
^{\alpha -1}dt &=&\frac{2\left( V_{2}-y\right) ^{\alpha +1}}{\alpha \left(
\alpha +1\right) }+\left( V_{2}-V_{1}\right) ^{\alpha }\left[ \frac{%
V_{2}-V_{1}}{\alpha +1}-\frac{V_{2}-y}{\alpha }\right]
\end{eqnarray*}%
and%
\begin{equation*}
\dint\limits_{V_{2}}^{b}\left\vert z-t\right\vert \left( b-t\right) ^{\alpha
-1}dt=\left( b-V_{2}\right) ^{\alpha }\left[ \frac{b-z}{\alpha }-\frac{%
b-V_{2}}{\alpha +1}\right] .
\end{equation*}

\item If $x\leq V_{1}\leq V_{2}\leq y\leq z$, then we have%
\begin{eqnarray*}
\dint\limits_{a}^{V_{1}}\left\vert x-t\right\vert \left( t-a\right) ^{\alpha
-1}dt &=&\frac{2\left( x-a\right) ^{\alpha +1}}{\alpha \left( \alpha
+1\right) }+\left( V_{1}-a\right) ^{\alpha }\left[ \frac{V_{1}-a}{\alpha +1}-%
\frac{x-a}{\alpha }\right] , \\
\dint\limits_{V_{1}}^{V_{2}}\left\vert y-t\right\vert \left( V_{2}-t\right)
^{\alpha -1}dt &=&\left( V_{2}-V_{1}\right) ^{\alpha }\left[ \frac{y-V_{2}}{%
\alpha }+\frac{V_{2}-V_{1}}{\alpha +1}\right] ,
\end{eqnarray*}%
and%
\begin{equation*}
\dint\limits_{V_{2}}^{b}\left\vert z-t\right\vert \left( b-t\right) ^{\alpha
-1}dt=\frac{2\left( b-z\right) ^{\alpha +1}}{\alpha \left( \alpha +1\right) }%
+\left( b-V_{2}\right) ^{\alpha }\left[ \frac{b-V_{2}}{\alpha +1}-\frac{b-z}{%
\alpha }\right] .
\end{equation*}

\item If $x\leq V_{1}\leq y\leq V_{2}\leq z$, then we have%
\begin{eqnarray*}
\dint\limits_{a}^{V_{1}}\left\vert x-t\right\vert \left( t-a\right) ^{\alpha
-1}dt &=&\frac{2\left( x-a\right) ^{\alpha +1}}{\alpha \left( \alpha
+1\right) }+\left( V_{1}-a\right) ^{\alpha }\left[ \frac{V_{1}-a}{\alpha +1}-%
\frac{x-a}{\alpha }\right] , \\
\dint\limits_{V_{1}}^{V_{2}}\left\vert y-t\right\vert \left( V_{2}-t\right)
^{\alpha -1}dt &=&\frac{2\left( V_{2}-y\right) ^{\alpha +1}}{\alpha \left(
\alpha +1\right) }+\left( V_{2}-V_{1}\right) ^{\alpha }\left[ \frac{%
V_{2}-V_{1}}{\alpha +1}-\frac{V_{2}-y}{\alpha }\right]
\end{eqnarray*}%
and%
\begin{equation*}
\dint\limits_{V_{2}}^{b}\left\vert z-t\right\vert \left( b-t\right) ^{\alpha
-1}dt=\frac{2\left( b-z\right) ^{\alpha +1}}{\alpha \left( \alpha +1\right) }%
+\left( b-V_{2}\right) ^{\alpha }\left[ \frac{b-V_{2}}{\alpha +1}-\frac{b-z}{%
\alpha }\right] .
\end{equation*}

\item If $x\leq V_{1}\leq y\leq z\leq V_{2}$, then we have%
\begin{eqnarray*}
\dint\limits_{a}^{V_{1}}\left\vert x-t\right\vert \left( t-a\right) ^{\alpha
-1}dt &=&\frac{2\left( x-a\right) ^{\alpha +1}}{\alpha \left( \alpha
+1\right) }+\left( V_{1}-a\right) ^{\alpha }\left[ \frac{V_{1}-a}{\alpha +1}-%
\frac{x-a}{\alpha }\right] , \\
\dint\limits_{V_{1}}^{V_{2}}\left\vert y-t\right\vert \left( V_{2}-t\right)
^{\alpha -1}dt &=&\frac{2\left( V_{2}-y\right) ^{\alpha +1}}{\alpha \left(
\alpha +1\right) }+\left( V_{2}-V_{1}\right) ^{\alpha }\left[ \frac{%
V_{2}-V_{1}}{\alpha +1}-\frac{V_{2}-y}{\alpha }\right]
\end{eqnarray*}%
and%
\begin{equation*}
\dint\limits_{V_{2}}^{b}\left\vert z-t\right\vert \left( b-t\right) ^{\alpha
-1}dt=\left( b-V_{2}\right) ^{\alpha }\left[ \frac{b-z}{\alpha }-\frac{%
b-V_{2}}{\alpha +1}\right] .
\end{equation*}

\item If $x\leq y\leq V_{1}\leq V_{2}\leq z$, then we have%
\begin{eqnarray*}
\dint\limits_{a}^{V_{1}}\left\vert x-t\right\vert \left( t-a\right) ^{\alpha
-1}dt &=&\frac{2\left( x-a\right) ^{\alpha +1}}{\alpha \left( \alpha
+1\right) }+\left( V_{1}-a\right) ^{\alpha }\left[ \frac{V_{1}-a}{\alpha +1}-%
\frac{x-a}{\alpha }\right] , \\
\dint\limits_{V_{1}}^{V_{2}}\left\vert y-t\right\vert \left( V_{2}-t\right)
^{\alpha -1}dt &=&\left( V_{2}-V_{1}\right) ^{\alpha }\left[ \frac{%
V_{2}-V_{1}}{\alpha +1}-\frac{V_{2}-y}{\alpha }\right] ,
\end{eqnarray*}%
and%
\begin{equation*}
\dint\limits_{V_{2}}^{b}\left\vert z-t\right\vert \left( b-t\right) ^{\alpha
-1}dt=\frac{2\left( b-z\right) ^{\alpha +1}}{\alpha \left( \alpha +1\right) }%
+\left( b-V_{2}\right) ^{\alpha }\left[ \frac{b-V_{2}}{\alpha +1}-\frac{b-z}{%
\alpha }\right] .
\end{equation*}

\item If $x\leq y\leq V_{1}\leq z\leq V_{2}$ or $x\leq y\leq z\leq V_{1}\leq
V_{2}$, then we have%
\begin{eqnarray*}
\dint\limits_{a}^{V_{1}}\left\vert x-t\right\vert \left( t-a\right) ^{\alpha
-1}dt &=&\frac{2\left( x-a\right) ^{\alpha +1}}{\alpha \left( \alpha
+1\right) }+\left( V_{1}-a\right) ^{\alpha }\left[ \frac{V_{1}-a}{\alpha +1}-%
\frac{x-a}{\alpha }\right] , \\
\dint\limits_{V_{1}}^{V_{2}}\left\vert y-t\right\vert \left( V_{2}-t\right)
^{\alpha -1}dt &=&\left( V_{2}-V_{1}\right) ^{\alpha }\left[ \frac{%
V_{2}-V_{1}}{\alpha +1}-\frac{V_{2}-y}{\alpha }\right] ,
\end{eqnarray*}%
and%
\begin{equation*}
\dint\limits_{V_{2}}^{b}\left\vert z-t\right\vert \left( b-t\right) ^{\alpha
-1}dt=\left( b-V_{2}\right) ^{\alpha }\left[ \frac{b-z}{\alpha }-\frac{%
b-V_{2}}{\alpha +1}\right] .
\end{equation*}%
Using the inequality (\ref{3-1a}) and the above identities $%
\int_{a}^{V_{1}}\left\vert x-t\right\vert \left( t-a\right) ^{\alpha -1}dt$, 
$\int_{V_{1}}^{V_{2}}\left\vert y-t\right\vert \left( V_{2}-t\right)
^{\alpha -1}dt$ and $\int_{V_{2}}^{b}\left\vert z-t\right\vert \left(
b-t\right) ^{\alpha -1}dt$, we derive the inequality (\ref{3-1}). This
completes the proof.
\end{enumerate}
\end{proof}

Under the assumptions of Theorem \ref{3.1}, we have the following
corollaries and remarks:

\begin{remark}
In Theorem \ref{3.1}, if we take $\alpha =1$, then then the inequality (\ref%
{3-1}) reduces the inequality (\ref{1-3}) in Theorem \ref{1.2}.
\end{remark}

\begin{corollary}
\label{3.2}In Theorem \ref{3.1}, let $\delta \in \left[ \frac{1}{2},1\right] 
$, $x=\delta a+(1-\delta )b$, $y=\frac{a+b}{2}$ and $z=(1-\delta )a+\delta b$%
. Then, we have the inequality%
\begin{equation}
\left\vert \lambda ^{\alpha }f(\delta a+(1-\delta )b)+\eta ^{\alpha }f(\frac{%
a+b}{2})+\mu ^{\alpha }f((1-\delta )a+\delta b)-\frac{\Gamma (\alpha +1)}{%
\left( b-a\right) ^{\alpha }}\left[ J_{V_{1}-}^{\alpha
}f(a)+J_{V_{1}+}^{\alpha }f(V_{2})+J_{V_{2}+}^{\alpha }f(b)\right]
\right\vert  \notag
\end{equation}%
\begin{equation*}
\leq \frac{MN\left( \alpha ,\lambda ,\eta ,\delta \right) \left( b-a\right) 
}{\alpha +1}
\end{equation*}%
where $N\left( \alpha ,\lambda ,\eta ,\delta \right) $ is defined as follows:
\end{corollary}

\begin{enumerate}
\item If $\lambda +\eta \leq 1-\delta $ or $\lambda \leq 1-\delta \leq
\lambda +\eta \leq \frac{1}{2}$, then%
\begin{eqnarray*}
N\left( \alpha ,\lambda ,\eta ,\delta \right) &=&\lambda ^{\alpha }\left[
\left( 1-\delta \right) \left( \alpha +1\right) -\alpha \lambda \right]
+\eta ^{\alpha }\left[ \left( \frac{1}{2}-\lambda -\eta \right) \left(
\alpha +1\right) +\alpha \eta \right] \\
&&+2\left( 1-\delta \right) ^{\alpha +1}+\left( 1-\lambda -\eta \right)
^{\alpha }\left[ \alpha \left( 1-\lambda -\eta \right) -\left( \alpha
+1\right) \left( 1-\delta \right) \right] .
\end{eqnarray*}

\item If $\lambda \leq 1-\delta \leq \frac{1}{2}\leq \lambda +\eta \leq
\delta $, then 
\begin{eqnarray*}
N\left( \alpha ,\lambda ,\eta ,\delta \right) &=&\lambda ^{\alpha }\left[
\left( 1-\delta \right) \left( \alpha +1\right) -\alpha \lambda \right]
+2\left( \lambda +\eta -\frac{1}{2}\right) ^{\alpha +1}+\eta ^{\alpha }\left[
\alpha \eta -\left( \alpha +1\right) \left( \lambda +\eta -\frac{1}{2}%
\right) \right] \\
&&+2\left( 1-\delta \right) ^{\alpha +1}+\left( 1-\lambda -\eta \right)
^{\alpha }\left[ \alpha \left( 1-\lambda -\eta \right) -\left( \alpha
+1\right) \left( 1-\delta \right) \right] .
\end{eqnarray*}

\item If $\lambda \leq 1-\delta \leq \frac{1}{2}\leq \delta \leq \lambda
+\eta $, then 
\begin{eqnarray*}
N\left( \alpha ,\lambda ,\eta ,\delta \right) &=&\lambda ^{\alpha }\left[
\left( 1-\delta \right) \left( \alpha +1\right) -\alpha \lambda \right]
+2\left( \lambda +\eta -\frac{1}{2}\right) ^{\alpha +1} \\
&&+\eta ^{\alpha }\left[ \alpha \eta -\left( \alpha +1\right) \left( \lambda
+\eta -\frac{1}{2}\right) \right] +\left( 1-\lambda -\eta \right) ^{\alpha }%
\left[ \left( \alpha +1\right) \left( 1-\delta \right) -\alpha \left(
1-\lambda -\eta \right) \right] .
\end{eqnarray*}

\item If $1-\delta \leq \lambda \leq \lambda +\eta \leq \frac{1}{2}$, then 
\begin{eqnarray*}
N\left( \alpha ,\lambda ,\eta ,\delta \right) &=&4\left( 1-\delta \right)
^{\alpha +1}+\lambda ^{\alpha }\left[ \alpha \lambda -\left( 1-\delta
\right) \left( \alpha +1\right) \right] +\eta ^{\alpha }\left[ \alpha \eta
+\left( \alpha +1\right) \left( \frac{1}{2}-\lambda -\eta \right) \right] \\
&&+\left( 1-\lambda -\eta \right) ^{\alpha }\left[ \alpha \left( 1-\lambda
-\eta \right) -\left( \alpha +1\right) \left( 1-\delta \right) \right] .
\end{eqnarray*}

\item If $1-\delta \leq \lambda \leq \frac{1}{2}\leq \lambda +\eta \leq
\delta $, then%
\begin{eqnarray*}
N\left( \alpha ,\lambda ,\eta ,\delta \right) &=&4\left( 1-\delta \right)
^{\alpha +1}+\lambda ^{\alpha }\left[ \alpha \lambda -\left( 1-\delta
\right) \left( \alpha +1\right) \right] +2\left( \lambda +\eta -\frac{1}{2}%
\right) ^{\alpha +1} \\
&&+\eta ^{\alpha }\left[ \alpha \eta -\left( \alpha +1\right) \left( \lambda
+\eta -\frac{1}{2}\right) \right] +\left( 1-\lambda -\eta \right) ^{\alpha }%
\left[ \alpha \left( 1-\lambda -\eta \right) -\left( \alpha +1\right) \left(
1-\delta \right) \right] .
\end{eqnarray*}

\item If $1-\delta \leq \lambda \leq \frac{1}{2}\leq \delta \leq \lambda
+\eta $, then%
\begin{eqnarray*}
N\left( \alpha ,\lambda ,\eta ,\delta \right) &=&2\left( 1-\delta \right)
^{\alpha +1}+\lambda ^{\alpha }\left[ \alpha \lambda -\left( 1-\delta
\right) \left( \alpha +1\right) \right] +2\left( \lambda +\eta -\frac{1}{2}%
\right) ^{\alpha +1} \\
&&+\eta ^{\alpha }\left[ \alpha \eta -\left( \alpha +1\right) \left( \lambda
+\eta -\frac{1}{2}\right) \right] +\left( 1-\lambda -\eta \right) ^{\alpha }%
\left[ \left( \alpha +1\right) \left( 1-\delta \right) -\alpha \left(
1-\lambda -\eta \right) \right] .
\end{eqnarray*}

\item If $\frac{1}{2}\leq \lambda \leq \lambda +\eta \leq \delta $, then%
\begin{eqnarray*}
N\left( \alpha ,\lambda ,\eta ,\delta \right) &=&4\left( 1-\delta \right)
^{\alpha +1}+\lambda ^{\alpha }\left[ \alpha \lambda -\left( 1-\delta
\right) \left( \alpha +1\right) \right] +\eta ^{\alpha }\left[ \alpha \eta
-\left( \alpha +1\right) \left( \lambda +\eta -\frac{1}{2}\right) \right] \\
&&+\left( 1-\lambda -\eta \right) ^{\alpha }\left[ \alpha \left( 1-\lambda
-\eta \right) -\left( \alpha +1\right) \left( 1-\delta \right) \right] .
\end{eqnarray*}

\item If $\frac{1}{2}\leq \lambda \leq \delta \leq \lambda +\eta $ or $%
\delta \leq \lambda $, then%
\begin{eqnarray*}
N\left( \alpha ,\lambda ,\eta ,\delta \right) &=&2\left( 1-\delta \right)
^{\alpha +1}+\lambda ^{\alpha }\left[ \alpha \lambda -\left( 1-\delta
\right) \left( \alpha +1\right) \right] \\
&&+\eta ^{\alpha }\left[ \alpha \eta -\left( \alpha +1\right) \left( \lambda
+\eta -\frac{1}{2}\right) \right] +\left( 1-\lambda -\eta \right) ^{\alpha }%
\left[ \left( \alpha +1\right) \left( 1-\delta \right) -\alpha \left(
1-\lambda -\eta \right) \right] .
\end{eqnarray*}
\end{enumerate}

\begin{corollary}
In Corollary \ref{3.2}, if we take $\delta =1$, $\lambda =\mu =\frac{\theta 
}{2}$ and $\eta =1-\theta $ with $\theta \in \left[ 0,1\right] $, then we
have the following weighted Bullen-type inequality for $M$-Lipschitzian
functions via fractional integrals%
\begin{equation}
\left\vert \left( \frac{\theta }{2}\right) ^{\alpha }\left(
f(a)+f((b)\right) +\left( 1-\theta \right) ^{\alpha }f(\frac{a+b}{2})+-\frac{%
\Gamma (\alpha +1)}{\left( b-a\right) ^{\alpha }}\left[ J_{V_{1}-}^{\alpha
}f(a)+J_{V_{1}+}^{\alpha }f(V_{2})+J_{V_{2}+}^{\alpha }f(b)\right]
\right\vert  \notag
\end{equation}%
\begin{equation}
\leq \frac{M\left[ 2\alpha \left( \frac{\theta }{2}\right) ^{\alpha
+1}+\left( 1-\theta \right) ^{\alpha +1}\frac{\alpha -1}{2}+2\left( \frac{%
1-\theta }{2}\right) ^{\alpha +1}\right] \left( b-a\right) }{\alpha +1}.
\label{3-2a}
\end{equation}
\end{corollary}

\begin{remark}
In the inequality (\ref{3-2a}), if we take $\theta =\frac{1}{2}$, then the
inequality (\ref{3-2a}) reduces to the following Bullen-type inequality for $%
M$-Lipschitzian functions via fractional integrals%
\begin{equation}
\left\vert \frac{1}{2}\left[ \frac{f(a)+f(b)}{2}+f\left( \frac{a+b}{2}%
\right) \right] -\frac{2^{\alpha -1}\Gamma (\alpha +1)}{\left( b-a\right)
^{\alpha }}\left[ J_{\left( \frac{3a+b}{4}\right) -}^{\alpha }f(a)+J_{\left( 
\frac{3a+b}{4}\right) +}^{\alpha }f(\frac{a+3b}{4})+J_{\left( \frac{a+3b}{4}%
\right) +}^{\alpha }f(b)\right] \right\vert  \notag
\end{equation}%
\begin{equation*}
\leq \frac{M\left( b-a\right) }{2^{\alpha +2}\left( \alpha +1\right) }\left[
\alpha +1+2^{\alpha -1}\left( \alpha -1\right) \right] .
\end{equation*}
\end{remark}

\begin{remark}
In the inequality (\ref{3-2a}), if we take $\theta =\frac{1}{3}$, then the
inequality (\ref{3-2a}) reduces to the following Simpson-type inequality for 
$M$-Lipschitzian functions via fractional integrals%
\begin{equation}
\left\vert \frac{1}{6}\left[ f(a)+4f\left( \frac{a+b}{2}\right) +f(b)\right]
-\frac{6^{\alpha -1}\Gamma (\alpha +1)}{\left( b-a\right) ^{\alpha }}\left[
J_{\left( \frac{5a+b}{6}\right) -}^{\alpha }f(a)+J_{\left( \frac{5a+b}{6}%
\right) +}^{\alpha }f(\frac{a+5b}{6})+J_{\left( \frac{a+5b}{6}\right)
+}^{\alpha }f(b)\right] \right\vert  \notag
\end{equation}%
\begin{equation*}
\leq \frac{M\left( b-a\right) }{18\left( \alpha +1\right) }\left[ \alpha
+2^{2\alpha }\left( \alpha -1\right) 3+2^{\alpha +1}\right]
\end{equation*}
\end{remark}

\end{document}